\documentclass[12pt]{amsart}
\usepackage{latexsym,amssymb,amscd}

\def\NZQ{\Bbb}               

\def\ZZ{{\NZQ Z}}

\def\B'c{{\mathcal{B'}}}
\def\U'c{{\mathcal{U'}}}

\def\opn#1#2{\def#1{\operatorname{#2}}} 
\opn\chara{char}
\opn\length{\ell}
\opn\projdim{proj\,dim}
\opn\injdim{inj\,dim}
\opn\ini{in}
\opn\rank{rank}
\opn\depth{depth}
\opn\sdepth{sdepth}
\opn\height{ht}
\opn\embdim{emb\,dim}
\opn\codim{codim}

\opn\Tr{Tr}
\opn\bigrank{big\,rank}
\opn\superheight{superheight}\opn\lcm{lcm}
\opn\trdeg{tr\,deg}%
\opn\reg{reg}
\opn\lreg{lreg}
\opn\set{set}
\opn\supp{Supp}
\opn\shad{Shad}
\opn\div{div}
\opn\Div{Div}
\opn\cl{cl}
\opn\Cl{Cl}
\opn\Spec{Spec}
\opn\Supp{Supp}
\opn\supp{supp}
\opn\Sing{Sing}
\opn\Ass{Ass}
\opn\Min{Min}
\opn\Ann{Ann}
\opn\Rad{Rad}
\opn\Soc{Soc}
\opn\Ker{Ker}
\opn\Coker{Coker}
\opn\Im{Im}
\opn\Hom{Hom}
\opn\Tor{Tor}
\opn\Ext{Ext}
\opn\End{End}
\opn\Aut{Aut}
\opn\id{id}

\opn\nat{nat}
\opn\GL{GL}
\opn\SL{SL}
\opn\mod{mod}
\opn\ord{ord}
\opn\aff{aff}
\opn\con{conv}
\opn\relint{relint}
\opn\st{st}
\opn\lk{lk}
\opn\cn{cn}
\opn\core{core}
\opn\vol{vol}
\opn\gr{gr}

\def\pot#1#2{#1[\kern-0.28ex[#2]\kern-0.28ex]}

\opn\dirlim{\underrightarrow{\lim}}
\opn\invlim{\underleftarrow{\lim}}
\def\pnt{{\raise0.5mm\hbox{\large\bf.}}}

\def\Implies{\ifmmode\Longrightarrow \else
     \unskip${}\Longrightarrow{}$\ignorespaces\fi}
\def\implies{\ifmmode\Rightarrow \else
     \unskip${}\Rightarrow{}$\ignorespaces\fi}
\def\iff{\ifmmode\Longleftrightarrow \else
     \unskip${}\Longleftrightarrow{}$\ignorespaces\fi}

\let\:=\colon
\newtheorem{Theorem}{Theorem}[section]
\newtheorem{Lemma}[Theorem]{Lemma}
\newtheorem{Corollary}[Theorem]{Corollary}
\newtheorem{Proposition}[Theorem]{Proposition}
\newtheorem{Remark}[Theorem]{Remark}

\newtheorem{Example}[Theorem]{Example}

\let\epsilon=\varepsilon
\let\phi=\varphi
\let\kappa=\varkappa
\numberwithin{equation}{section}

\title{Values and bounds of the Stanley depth}
\author{Muhammad Ishaq}
\thanks{The author would like to express his gratitude to ASSMS of GC University Lahore and the Institute of Mathematics of Romanian Academy for creating a very appropriate atmosphere for research work. This research is partially supported by HEC Pakistan}
\address{Muhammad Ishaq, Abdus Salam School of Mathematical Sciences, GC University, Lahore, 68-B New Muslim Town Lahore, Pakistan.}
\email{ishaq\_\,maths@yahoo.com}

\begin{document}

\maketitle

\begin{abstract}
We give different bounds for the Stanley depth of a monomial ideal $I$ of a polynomial algebra $S$ over a field $K$. For example we show that the Stanley depth of $I$ is less or equal with the Stanley depth of any prime ideal associated to $S/I$. Also we show that the Stanley conjecture holds for $I$ and $S/I$ when the associated prime ideals of $S/I$ are generated by disjoint sets of variables.
2000 Mathematics Classification: Primary 13H10, Secondary 13P10, 13C14, 13F20.
\end{abstract}

\section*{Introduction}
Let $K$ be a field, $S = K[x_1,\dots,x_n]$ be the polynomial ring in n variables with coefficients in $K$ and $M$ be a finitely generated $\ZZ^n$-graded $S$-module. Let $u\in M$ be a homogeneous element in $M$ and $Z$ a subset of the set of variables $Z\subset\{x_1,\dots,x_n\}$. We denote by $uK[Z]$ the $K$-subspace of $M$ generated by all elements $uv$ where $v$ is a monomial in $K[Z]$. If $uK[Z]$ is a free $K[Z]$-module, the $\ZZ^n$-graded $K$-space $uK[Z]\subset M$ is called a Stanley space of dimension $|Z|$.
A Stanley decomposition of M is a presentation of the $\ZZ^n$-graded $K$-vector space $M$ as a finite direct sum of Stanley spaces $$\mathcal{D}:M=\bigoplus\limits_{i=1}^{s}u_iK[Z_i].$$
The number
$$\sdepth\mathcal{D}=\min\{|Z_i| : i =1,\dots,s\}$$ is called the Stanley depth of decomposition $\mathcal{D}$ and the number
$$\sdepth M := \max\{\sdepth\mathcal{D} : \mathcal{D} \hbox{ \ is a Stanley decomposition of $M$}\}$$
is called the Stanley depth of M. This is a combinatorial invariant and does not depend on the characteristic of $K$.
The following widely open conjecture is due to Stanley \cite{RP}:
$$\depth M\leq \sdepth M \hbox{ \ for all $\ZZ^n$-graded $S$-modules $M$}.$$
\indent Let $P$ be an associated prime ideal of $S/I$. It is well known that $\depth_SS/I\leq \depth_S S/P=\dim S/ P$ and  so $\depth_SI\leq \depth_S P$. By Apel \cite{JP}
we have also $\sdepth_S S/I\leq \dim S/P$. Moreover our Theorem 1.1 says that it holds also  $\sdepth_SI\\\leq \sdepth_S P$. Let $G(I)$  be the minimal monomial generators of $I$ and $r=|G(I)|$. If there exists an associated prime ideal $P$ of $S/I$ such that $\height P=r$ then $\sdepth_SI=n-\lfloor\frac{r}{2}\rfloor$ as says our Corollaries \ref{t}, \ref{t1}. Let $\Ass S/I=\{P_1,\ldots, P_s\}$. If $P_i\not\subset \sum\limits_{j\neq i}^sP_j$ for all $1\leq i\leq s$ and $I$ is squarefree then $\sdepth_S(I)\geq \depth_S(I)$ by \cite[Theorem 2.3]{D3}. Suppose that $G(P_i)\cap G(P_j)=\emptyset$ for all $i\neq j$. In particular the above result holds in this frame as it was stated by \cite[Theorem 1.4]{AD}. Our Corollary \ref{1bb} shows that the last result holds even when $I$ is not squarefree. Moreover, we have $\sdepth_S(S/I)\geq \depth_S(S/I)$ by our Theorem \ref{1c}. Hence The Stanley conjecture holds in this frame for $I$ and $S/I$. \\ \indent It is hard to compute Stanley depth even using the nice method from \cite{HVZ} and so it is really important to give at least tight bounds (see for example \cite{CI}). If $s=3$, $r=\height P_1\leq e=\height P_2\leq q=\height P_3$ then an upper bound  for $\sdepth_SS/I$ is given by $e+\lceil\frac{q}{2}\rceil$ even $r+\lfloor\frac{q}{2}\rfloor$ except possible in the case when $r$ is even and $e=r+1$ (see Corollary \ref{as} and Proposition \ref{p}). On the other hand, $\sdepth_SS/I\geq \min\{r+e,r+\lfloor\frac{q}{2}\rfloor,  \lfloor\frac{e}{2}\rfloor+\lfloor\frac{q}{2}\rfloor\}$ as says our Lemma \ref{bs}. Section 3 is devoted to find good upper bounds for the $\sdepth_S I$ when  $s=3$ but $(G(P_i))_i$ are not necessarily disjoint. These bounds are not pleasant (see Proposition \ref{12}, Theorem \ref{2}, Proposition \ref{5}) but very tight in certain cases (see our examples \ref{315},\ref{310}, \ref{317}, \ref{312}).
Sometimes we are able to give some values of Stanley depth as in Corollaries \ref{t}, \ref{t1}, \ref{value}.
\section{upper bounds of the stanley depth of monomial ideals}
One of our main result is the following:
\begin{Theorem}\label{f}
Let $I\subset S$ be a monomial ideal such that $Ass(S/I)=\{P_1,\dots,P_s\}$. Then $$\sdepth(I)\leq \min\{\sdepth (P_i), \ 1\leq i \leq s\}.$$
\end{Theorem}
\begin{proof}
Let $P_i\in \Ass(S/I)$ then $P_i$ is monomial and there exists a monomial $w_i\notin I$ such that $I:w_i=P_i$. By \cite[Proposition 1.3]{D2} (see arXiv version) we have $\sdepth(I)\leq \sdepth(I:w_i)=\sdepth(P_i)$. Thus we get $$\sdepth(I)\leq \min\{\sdepth (P_i), \ 1\leq i \leq s\}.$$
\end{proof}
\begin{Corollary}
Let $I\subset S$ be a monomial ideal such that $\mathfrak{m}\in \Ass(S/I)$, then $\sdepth(I)\leq \lceil\frac{n}{2}\rceil$.
\end{Corollary}
\begin{Corollary}\label{t}
Let $I\subset S$ be a monomial ideal with $|G(I)|=m$ suppose that $m$ is even, and let there exists a prime ideal $P\in \Ass(S/I)$ such that $\height(P)=m$. Then $$\sdepth_S(I)=n-\frac{m}{2}.$$
\end{Corollary}
\begin{proof}
By \cite[Theorem 2.3]{O} we have $\sdepth(I)\geq n-\frac{m}{2}$. Since there exists a prime ideal $P\in \Ass (S/I)$, with $\height(P)=m$, thus by Theorem \ref{f} we have that $\sdepth(I)\leq n-\frac{m}{2}$.
\end{proof}
\begin{Corollary}\label{t1}
Let $I\subset S$ be a monomial ideal with $|G(I)|=m$ suppose that $m$ is odd, and let there exists a prime ideal $P\in \Ass(S/I)$ such that $\height(P)\geq m-1$. Then $$\sdepth_S(I)=n-\lfloor\frac{m}{2}\rfloor.$$
\end{Corollary}
\begin{Corollary}\label{t2}
Let $I\subset S$ be a monomial ideal and let $P_i\in \Ass(S/I)=\{P_1,\dots,P_s\}$, $d_i=\height(P_i)$, let $I':=(I,x_{n+1},x_{n+2})\subset S':=S[x_{n+1},x_{n+2}]$, Then $$\sdepth_{S'}(I')\leq \min\{n+1-\lfloor\frac{d_i}{2}\rfloor, \ 1\leq i \leq s\}.$$
\end{Corollary}
\begin{proof}
We have $\Ass_{S'}(S'/I')=\{(P_1,x_{n+1},x_{n+2}),\dots,(P_s,x_{n+1},x_{n+2})\}$. Using Theorem \ref{f} we get that
$\sdepth_{S'}(I')\leq \min\{\sdepth (P_i,x_{n+1},x_{n+2}), \ 1\leq i \leq s\}$ and it enough to apply \cite[Theorem 1.3]{C1}.
\end{proof}

\begin{Remark}\label{r}
{\em By \cite[Lemma 2.11]{MI} we have $\sdepth_{S'}(I')\leq \sdepth_S(I)+2$ but our above corollary says that the bound of $\sdepth_{S'}(I')$ given by Theorem \ref{f} is $1+$ the bound of $\sdepth_S(I)$.
}
\end{Remark}
Next we need the following:
\begin{Lemma}\label{sdep1}\cite[Lemma 1.2]{AD} Let $S=K[x_1,\ldots,x_n]$ and $I \subset K[x_1,\ldots,x_r]=S'$, $J
\subset K[x_{r+1},\ldots,x_n]=S''$, where $1<r<n$ be monomial ideals. Then
$$\sdepth_S(IS \cap JS) \ge \sdepth_{S'}(I)+\sdepth_{S''}(J).$$
\end{Lemma}
\begin{Corollary}\label{1a}
Let $I_i\subset S_i=K[Z_i]$, $i=1,\dots,r$ be monomial ideals where $Z_i\subset \{x_1,\dots,x_n\}$ and $Z_i\cap Z_j=\emptyset$ for all $i\neq j$ and $m_i=|G(I_i)|$. Let $S=K[Z_1\cup Z_2\cup\dots\cup Z_r]$ and  $\sum\limits_{i=1}^r|Z_i|=n$. Then $$\sdepth_S(I_1S\cap I_2S\cap\dots\cap I_rS)\geq \sdepth_{S_1}(I_1)+\dots+\sdepth_{S_r}(I_r)\geq n-\sum\limits_{i=1}^r\lfloor\frac{m_i}{2}\rfloor.$$
\end{Corollary}
\begin{Remark}
{\em
With the hypothesis from Corollary \ref{1a}, if the Stanley conjecture holds for all $I_i\subset S_i$, then the Stanley conjecture also holds for $I_1S\cap I_2S\cap\dots \cap I_rS$ similarly as in \cite[Theorem 1.4]{AD}.}
\end{Remark}
\begin{Example}
{\em Let $I=(x_1x_3,x_1x_4,x_1x_5,x_1x_6,x_2)\cap (x_7,\dots,x_{11})$ and $S=K[x_1,\\\dots,x_{11}]$. We know by \cite[\!Example 2.20]{MI} that $$\sdepth((x_1x_3,x_1x_4,x_1x_5,x_1x_6,x_2)K[x_1,\dots\!,x_6])=4$$ and $\sdepth((x_7,\dots,x_{11})K[x_7,\dots,x_{11}])=3$. Then by Corollary \ref{1a}, $\sdepth(I)\geq 7$.
}
\end{Example}
\begin{Proposition}\label{1b}
Let $I\subset S$ be a monomial ideal and let $\Min(S/I)=\{P_1,\dots,P_s\}$. Assume that $P_i\nsubseteq \sum_{i=1}^{s-1}P_i$ for all $i\in [s]$.  Then $\depth(I)\leq s$ and $\depth(S/I)\leq s-1.$
\end{Proposition}
\begin{proof}
We have $\sqrt{I}=\bigcap\limits_{i=1}^sP_i$. By \cite[Theorem 2.3]{D3} and \cite[Theorem 2.6]{HTT} we get $\depth(I)\leq \depth(\sqrt{I})=s$ and consequently $\depth(S/I)\leq s-1.$
\end{proof}
\indent The following corollary extends \cite[Theorem 1.4]{AD}.
\begin{Corollary}\label{1bb}
Let $I\subset S$ be a monomial ideal, $\Ass(S/I)=\{P_1,\dots,P_m\}$ and suppose that $G(P_i)\cap G(P_j)=\emptyset$ for $i\neq j$. Then $\sdepth(I)\geq m.$ In particular Stanley's conjecture holds for $I$.
\begin{proof}
The proof follows from Proposition \ref{1b} and Corollary \ref{1a}.
\end{proof}
\end{Corollary}
\section{Stanley depth of multigraded cyclic modules}
The Corollary \ref{1bb} has an analogous form for $S/I$.
\begin{Theorem}\label{1c}
Let $I\subset S$ be a monomial ideal such that $I=\bigcap\limits_{i=1}^mQ_i$ is a reduced primary decomposition of $I$ and $P_i=\sqrt{Q_i}$ with $G(P_i)\cap G(P_j)=\emptyset$ for $i\neq j$. Then $\sdepth(S/I)\geq m-1.$ In particular Stanley's conjecture holds for $S/I$.
\end{Theorem}
\begin{proof}
 Let $\sum_{i=1}^{m-1}P_i=(x_1,\dots,x_{r})$, $S'=K[x_1,\dots,x_r]$. First we use induction on $m$. Case $m=1$ is clear. Fix $m>1$ and apply induction on $n-r$. Let $k$ be the minimum such that $x_n^{k}\in Q_m$. Define $I_i$ by $I\cap x_n^iS_1=x_n^iI_i$ for some ideal $I_i\subset S_1=K[x_1,\dots,x_{n-1}]$. Then $$S/I=S_1/I_0\oplus x_1(S_1/I_1)\oplus \dots \oplus x_n^k(S_1/I_k)[x_n].$$ Let $I'=Q_1\cap Q_2\cap \dots\cap Q_{m-1}$, If $n-r=1$, then $Q_m$ is an $(x_n)$-primary ideal and so it is given by a power of $x_n$, that is $(x_n^t)$. By \cite[Theorem 1.1]{C2} and by induction on $m$ we have $\sdepth_S S/(I'\cap x_n^t)=\sdepth S/(x_n^tI')=\sdepth S/I'=\sdepth(S'/I')+1\geq m-2+1=m-1$. Let $n-r>1$, by induction we have $\sdepth_{S_1}(S_1/I_i)\geq m-1$ for all $i<k$ and by \cite[Lemma 3.6]{HVZ} we have $\sdepth_{S_1}(S_1/I_k)\geq m-2+n-r-1\geq m-2$ since $r<n$. Then $$\sdepth_S(S/I)\geq \min \Big\{\{\sdepth_{S_1}(S_1/I_i)\}_{i=0,1,\dots, k-1},1+\sdepth_{S_1}(S_1/I_k)\Big\}.$$ If the minimum is reached on $1+\sdepth_{S_1}(S_1/I_k)$ then we are done because we get $\sdepth_S(S/I)\geq m-1$. If the minimum is reached on $\sdepth_{S_1}(S_1/I_i)$ for some $0\leq i<k$ then $\sdepth(S/I)\geq m-1$ again. Now by Proposition \ref{1b} we have $\depth(S/I)\leq m-1$ this implies that $\sdepth(S/I)\geq \depth(S/I)$.
\end{proof}
\indent Next we give we give upper and lower bounds for the Stanley depth of $S/I$ with $I=P_1\cap P_2\cap P_3$ the unique irredundent presentation of $I$ as the intersection of its minimal monomial prime ideals. By \cite[Lemma 3.6]{HVZ} it is enough to consider that $P_1+P_2+P_3=\mathfrak{m}.$ Let $\mathcal{D}:S/I=\bigoplus\limits_{i=1}^mu_iK[Z_i]$ be a Stanley decomposition. Then $Z_i$ cannot have in the same time variables from all $G(P_i)$, otherwise $u_iK[Z_i]$ will not be a free $K[Z_i]$-module.
\begin{Lemma}\label{ss}
Let $\mathcal{D}:S/I=\bigoplus\limits_{i=1}^mu_iK[Z_i]$ be a Stanley decomposition of $S/I$. Suppose that $u_1=1$ and $Z_1\subset \Big(G(P_1)\cup G(P_2)\Big)\backslash G(P_3)$. Then $$\sdepth(\mathcal{D})\leq \max\Big\{\dim(S/(P_2+P_3)),\dim(S/(P_1+P_3))\Big\}+\lceil\frac{\height(P_3)-t}{2}\rceil,$$ where $t=|G(P_1)\cap G(P_2)\cap G(P_3)|$.
\end{Lemma}
\begin{proof}
Let $Z:=G(P_3) \backslash \Big(G(P_1)\cap G(P_2)\Big)$ and $$\psi:P_3\cap K[Z]\hookrightarrow S/{I}$$ be the inclusion given by $$K[Z]\hookrightarrow S/{I}.$$ Then $P_3\cap K[Z]=\bigoplus\limits_i\psi^{-1}(u_iK[Z_i])$. If $\psi^{-1}(u_iK[Z_i])\neq 0$ implies there exists $u_if\in u_iK[Z_i]$ with $u_if\in P_3\cap K[Z]$. Since all the variables of $Z_1$ are in $P_1+P_2$ then $u_iK[Z_i]\cap P_3\cap K[Z]\neq 0$ implies $u_i\neq1$ and so $u_i\in P_3\cap K[Z]$ because $P_3$ gives maximal ideal in $K[Z]$. Let $Z_i'=Z_i\cap Z$. Then $\psi^{-1}(u_iK[Z_i])=u_iK[Z_i']$ and we get a Stanley decomposition of $P_3'=P_3\cap K[Z]$ by $P_3'=\bigoplus\limits_{u_i\in P_3'}u_iK[Z_i']$. It follows $|Z_i'|\leq \sdepth(P_3')=\lceil\frac{|Z|}{2}\rceil$. But either $Z_i\subset \Big(G(P_3)\cup G(P_1)\Big)\backslash G(P_2)$ or $Z_i\subset \Big(G(P_3)\cup G(P_2)\Big)\backslash G(P_1)$. In the first case we have $Z_i\subset\Big(G(P_3)\backslash G(P_2)\Big)\cup \Big(G(P_1)\backslash \Big(G(P_2)\cup G(P_3)\Big)\Big)$ and we get $Z_i\subset Z_i'\cup \Big(G(P_1)\backslash \Big(G(P_2)\cup G(P_3)\Big)\Big)$. It follows $|Z_i|\leq \dim(S/(P_2+P_3))+\lceil\frac{|Z|}{2}\rceil$. Similarly for the case $Z_i\subset \Big(G(P_3)\cup G(P_2)\Big)\backslash G(P_1)$ we get $|Z_i|\leq \dim(S/(P_1+P_3))+\lceil\frac{|Z|}{2}\rceil$. Thus we have $$\sdepth(\mathcal{D})\leq \max\Big\{\dim(S/(P_2+P_3)),\dim(S/(P_1+P_3))\Big\}+\lceil\frac{|Z|}{2}\rceil.$$ As $|Z|=\height(P_3)-t$ we are done.
\end{proof}
\begin{Corollary}\label{as}
Let $r\leq e\leq q$ be some positive integers with $r+e+q=n$ $P_1=(x_1,\dots,x_r)$, $P_2=(x_{r+1},\dots,x_{r+e})$, $P_3=(x_{r+e+1},\dots,x_{r+e+q})$ prime ideals of $S=K[x_1,\dots,x_n]$ and $I=P_1\cap P_2\cap P_3$. Then $$\sdepth(S/I)\leq e+\lceil\frac{q}{2}\rceil.$$
\end{Corollary}
This bound can be improved by the following proposition:
\begin{Proposition}\label{p}
Let $r\leq e\leq q$ be some positive integers with $r+e+q=n$, $P_1=(x_1,\dots,x_r)$, $P_2=(x_{r+1},\dots,x_{r+e})$, $P_3=(x_{r+e+1},\dots,x_{r+e+q})$ prime ideals of $S=K[x_1,\dots,x_n]$ and $I=P_1\cap P_2\cap P_3$.
Then $$\sdepth(S/I)\leq r+\min\{e,\lceil\frac{q}{2}\rceil\},$$ except in the case when $e=r+1$ and $r$ is odd.
\end{Proposition}
\begin{proof}
Let $r=1$, then by \cite[Theorem 1.1(1)]{C2} we have  $\sdepth(S/I)=\sdepth(S/(I:x_1))=1+\sdepth_{S'}S'/(P_2\cap P_3)=1+\min\{e,\lceil\frac{q}{2}\rceil\}$ so the inequality holds in this case. Now let $P_1=(P_1',x_2,\dots,x_r)$ where $P_1'=(x_1)$ and $I'=P_1'\cap P_2 \cap P_3$ i.e $I'S=I\cap (P_1'S)$. Let us consider the following exact sequence: $$0\longrightarrow I/I'S\longrightarrow S/I'S\longrightarrow S/I\longrightarrow 0,$$ by \cite[Lemma 2.2]{R1} we have $$\sdepth(S/I'S)\geq \min\{\sdepth(I/I'S),\sdepth(S/I)\}.$$ Since $I/I'S\simeq I/(I\cap P_1')S\simeq (I+P_1')/P_1'\simeq I\cap K[x_2,\dots,x_n]=(x_2,\dots,x_n)\cap P_2\cap P_3$ by using Corollary \ref{1a} we have $\sdepth (I/I'S)\geq \lceil\frac{r-1}{2}\rceil+\lceil\frac{e}{2}\rceil+\lceil\frac{q}{2}\rceil$. Now since $\sdepth(S/I'S)\leq 1+r-1+\min\{e,\lceil\frac{q}{2}\rceil\}=r+\min\{e,\lceil\frac{q}{2}\rceil\}$. Now we see that $\lceil\frac{r-1}{2}\rceil+\lceil\frac{e}{2}\rceil+\lceil\frac{q}{2}\rceil>r+\min\{e,\lceil\frac{q}{2}\rceil\}$ for all cases except $r=e$, $e=r+1$ and  $r$ is odd. Thus with these exceptions we get $\sdepth(S/I'S)\geq \sdepth(S/I)$. It follows $\sdepth(S/I)\leq r+\min\{e,\lceil\frac{q}{2}\rceil\}$ except in the cases when $r=e$, and $e=r+1$ and $r$ is odd. But when $r=e$ we can apply Corollary \ref{as}.
\end{proof}
\begin{Remark}
{\em
Let $I\subset S$ be a monomial ideal and $\Min(I)=\{P_1,P_2,P_3\}$, such that $G(P_i)\cap G(P_j)=\emptyset$ for all $i\neq j$. Then by \cite[Corollary 2.2]{MI} $\sdepth(S/I)\leq \sdepth (S/(P_1\cap P_2\cap P_3))$ and the upper bounds in Corollary \ref{as} and Proposition \ref{p} (with exceptions stated in the proposition) are also upper bounds for the Stanley depth of $S/I$.
}
\end{Remark}
\begin{Lemma}\label{bs}
Let $1\leq r\leq e\leq q$ be some positive integers with $r+e+q=n$, $P_1=(x_1,\dots,x_r)$, $P_2=(x_{r+1},\dots,x_{r+e})$, $P_3=(x_{r+e+1},\dots,x_{r+e+q})$ primes ideals of $S=K[x_1,\dots,x_n]$ and $I=P_1\cap P_2\cap P_3$. Then $$\sdepth(S/I)\geq\min\{r+e,r+\lceil\frac{q}{2}\rceil,\lceil\frac{e}{2}\rceil+\lceil\frac{q}{2}\rceil\}.$$
\end{Lemma}
\begin{proof}
$S/I$ can be written as the direct sum of some multigraded modules:
 $$S/I=S/P_3\bigoplus P_3/(P_3\cap P_2)\bigoplus (P_3\cap P_2)/I,$$ and we have $$\sdepth(S/I)\geq \min \{\sdepth(S/P_3),\sdepth(P_3/(P_3\cap P_2)),\sdepth((P_3\cap P_2)/I)\}$$ where $\sdepth S/P_3=\sdepth K[x_1,\dots,x_{r+e}]=r+e$. Now since $P_3/(P_3\cap P_2)\simeq (P_3+P_2)/P_2\simeq P_3\cap K[x_1,\dots,x_r,x_{r+e+1},\dots,x_n]$ we get $\sdepth(P_3/(P_3\cap P_2))=r+\lceil\frac{q}{2}\rceil$. Also since $(P_3\cap P_2)/I)\simeq ((P_3\cap P_2)+P_1)/P_1\simeq (P_3\cap P_2)\cap K[x_{r+1},\dots,x_{n}]$ we have using \cite[Lemma 4.1]{DQ} $\sdepth((P_3\cap P_2)/I))\geq \lceil\frac{e}{2}\rceil+\lceil\frac{q}{2}\rceil$, which is enough.
\end{proof}
\begin{Corollary}\label{value}
With the hypothesis form the above lemma, suppose that $r\leq\lceil\frac{e}{2}\rceil$. Then $\sdepth(S/I)=r+\min\{e,\lceil\frac{q}{2}\rceil\}$.
\end{Corollary}
\section{upper bounds for intersection of three prime ideals}
Our Theorem \ref{f} gives an upper bound of the Stanley depth of any monomial ideal. But this bound is not so tight in general. In this section we give an upper bound for the Stanley depth of ideals whose minimal associated primes set consists of three prime ideals. These bounds are tighter than the bound given by Theorem \ref{f}. By \cite[Corollary 2.2]{MI} it is enough to find an upper bound for the Stanley depth of intersection of three minimal prime ideals.\\ \indent Let $I=P_1\cap P_2 \cap P_3$ where $P_1,P_2$ and $P_3$ are monomial prime ideals of $S$. Suppose that ${P_i}\not\subset {P_j}$ for all $i\neq j$. By \cite[Lemma 3.6]{HVZ} it is enough to consider that ${P_1}+{P_2}+{P_3}=\mathfrak{m}$. After renumbering the variables we can always assume that ${P_1}=(x_1,\dots,x_t)$, ${P_2}=(x_{s+1},\dots,x_{r})$ and ${P_3}=(x_{q+1},\dots,x_n,x_1,\dots,x_u)$, with $s\leq t$, $q\leq r$, $t>u\geq0$. We consider the following cases:
\begin{enumerate}
\item$u>0$, \ $t<q$.
\item$u=0, \ s=t, \ q=r$.
\item$u=0, \ s\leq t, \ q\leq r, \ s<r$.
\end{enumerate}
The importance of considering theses cases is that for the ideals discussed for instance in Case(2)(a particular form of Case(3)) there exists  a reasonable upper bound (Theorem \ref{2}) which is clear from Corollary \ref{1a}, and examples.\\

We start with a lemma very useful later in this section.
\begin{Lemma}\label{9}
Let $I'\subset S'=S[x_{n+1}]$ be a monomial ideal, $x_{n+1}$ being a new variable. If $I'\cap S\neq(0)$, then $\sdepth_S(I'\cap S)\geq \sdepth_{S[x_{n+1}]}I'-1$.
\end{Lemma}
\begin{proof}
Let $\mathcal{D}:I'=\bigoplus_{i=1}^r u_iK[Z_i]$ be a Stanley decomposition of $I'$ such that $\sdepth(I')=\sdepth \mathcal{D}$. We claim that $$I'\cap S=\bigoplus_{x_{n+1}\notin \supp(u_i)}u_iK[Z_i\backslash \{x_{n+1}\}], \ \hbox{which is enough}.$$ Let $w \in I'\cap S$ be a monomial. Then there exists $i$ such that $w\in u_iK[Z_i]$ and so $x_{n+1}\nmid u_i$, $w\in u_iK[Z_i \backslash \{x_{n+1}\}]$. Thus $"\subset"$ holds, the other inclusion being trivial. As $u_iK[Z_i\backslash \{x_{n+1}\}]\cap u_jK[Z_j\backslash \{x_{n+1}\}]\subset u_iK[Z_i]\cap u_jK[Z_j]=\emptyset$ for $i\neq j$, we are done.
\end{proof}
A part of the inequality in \cite[Lemma 2.11]{MI} follows from the above lemma.
\begin{Corollary}\label{10}
Let $I\subset S$ be a monomial ideal of $S$ and $I'=(I,x_{n+1})\subset S'=S[x_{n+1}]$. Then $\sdepth_{S'}(I')\leq \sdepth_S(I)+1.$
\end{Corollary}
\begin{Corollary}\label{11}
Let $I'={P}_1\cap {P}_2\cap {P}_3\subset S'=S[x_{n+1}]$ where ${P}_i's$ are prime monomial ideals with ${{P}_1}=(x_1,\dots,\dots,x_t)$, ${{P}_2}=(x_{u+1},\dots,x_{n})$ and  ${{P}_3}=(x_{t+1},\dots,x_n,x_{n+1},\\x_1,\dots,x_u)$. Then
 $$\sdepth_{S'}(I')\leq \sdepth({I})+1,$$ where ${I}=I'\cap S=(x_1,\dots,\dots,x_t)\cap (x_{u+1},\dots,x_{n})\cap (x_{t+1},\dots,x_n,x_1,\dots,x_u).$
\end{Corollary}
We recall the method of Herzog et al. \cite{HVZ} for computing the Stanley depth of a squarefree monomial ideal $I$ using posets. Let $G(I)=\{v_1,\dots,v_m\}$ be the set of minimal monomial generators of $I$. The characteristic poset of $I$ with respect to $h=(1,\dots,1)$  (see \cite{HVZ}), denoted by $\mathcal{P}_I^{h}$ is in fact the set $$\mathcal{P}_I^{h}=\{C\subset [n] \ | \ C  \ \text{contains supp($v_i$) for some $i$}\} \ ,$$ where  $\supp(v_i)=\{j:x_j|v_i\}\subseteq[n]:=\{1,\dots,n\}$. For every $A,B\in \mathcal{P}_I^{h}$ with $A\subseteq B$, define the interval $[A,B]$ to be $\{C\in \mathcal{P}_I^{h}:A\subseteq B\subseteq C\}$. Let $\mathcal{P}:\mathcal{P}_I^{h}=\cup_{i=1}^r[C_i,D_i]$ be a partition of $\mathcal{P}_I^{h}$, and for each $i$, let $c(i)\in \{0,1\}^n$ be the n-uple such that $\supp(x^{c(i)})=C_i$. Then there is a Stanley decomposition $\mathcal{D}(\mathcal{P})$ of $I$ $$\mathcal{D}(\mathcal{P}):I=\bigoplus_{i=1}^sx^{c{(i)}}K[\{x_k|k\in D_i\}].$$ By \cite{HVZ} we get that $$\sdepth(I)=\max\{\sdepth\mathcal{D}(P) \ | \ \mathcal{P} \text { is a partition of} \ \mathcal{P}_I^{h}\}.$$
\indent Now we consider a special type of ideals which belongs to the case (1). Let
 ${P_1}=(x_1,\dots,x_t)$, ${P_2}=(x_{u+1},\dots,x_n)$ and ${P_3}=(x_{t+1},\dots,x_n,x_1\dots,x_u)$ ba prime ideals of $S$, where $0<u<t$ and $I=P_1\cap P_2\cap P_3$. Then
\begin{Lemma}\label{8}
$$\sdepth(I)\leq 2+\frac{\Big(
                                                                  \begin{array}{c}
                                                                    n \\
                                                                    3 \\
                                                                  \end{array}
                                                                 \Big)-\Big(
                                                                                \begin{array}{c}
                                                                                  u \\
                                                                                  3 \\
                                                                                \end{array}
                                                                              \Big)-\Big(
                                                                                      \begin{array}{c}
                                                                                        t-u \\
                                                                                        3 \\
                                                                                      \end{array}
                                                                                    \Big)-\Big(
                                                                                            \begin{array}{c}
                                                                                              n-t \\
                                                                                              3 \\
                                                                                            \end{array}
                                                                                          \Big)}{\Big(
                                                                          \begin{array}{c}
                                                                            n \\
                                                                            2 \\
                                                                          \end{array}
                                                                        \Big)-\Big(
                                                                                \begin{array}{c}
                                                                                  u \\
                                                                                  2 \\
                                                                                \end{array}
                                                                              \Big)-\Big(
                                                                                      \begin{array}{c}
                                                                                        t-u \\
                                                                                        2 \\
                                                                                      \end{array}
                                                                                    \Big)-\Big(
                                                                                            \begin{array}{c}
                                                                                              n-t \\
                                                                                              2 \\
                                                                                            \end{array}
                                                                                          \Big)
                                                                 },
$$ where $\Big(
          \begin{array}{c}
            a \\
            b \\
          \end{array}
        \Big)=0$ when $a<b$.
\end{Lemma}
\begin{proof}
We follow the proof of \cite[Theorem 2.8]{MI}(see also \cite{KSSY}). Note that ${I}$ is generated by monomials of degree 2. Let $k:=\sdepth({I})$. The poset $P_{{I}}$ has a partition $\mathcal{P}$ : $P_{{I}}= \bigcup_{i=1}^{s}$ $[C_i,D_i]$ , satisfying $\sdepth(\mathcal{D}(\mathcal{P}))=k$ where $\mathcal{D(\mathcal{P})}$ is the Stanley decomposition of ${I}$ with respect to the partition $\mathcal{P}$. For each interval $[C_i,D_i]$ in $\mathcal{P}$ with $|C_i|=2$ we have $|D_i|\geq k$. Also there are $|D_i|-|C_i|$ subsets of cardinality $3$ in this interval. Since these intervals are disjoint , counting the number of subsets of cardinality 2 and 3 we have\\
\begin{multline*}
$$\Big[\Big(
                                                                          \begin{array}{c}
                                                                            n \\
                                                                            2 \\
                                                                          \end{array}
                                                                        \Big)-\Big(
                                                                                \begin{array}{c}
                                                                                  u \\
                                                                                  2 \\
                                                                                \end{array}
                                                                              \Big)-\Big(
                                                                                      \begin{array}{c}
                                                                                        t-u \\
                                                                                        2 \\
                                                                                      \end{array}
                                                                                    \Big)-\Big(
                                                                                            \begin{array}{c}
                                                                                              n-t \\
                                                                                              2 \\
                                                                                            \end{array}
\Big)\Big](k-2)\leq \\\Big( \begin{array}{c}
                                                                    n \\
                                                                    3 \\
                                                                  \end{array}
                                                                 \Big)-\Big(
                                                                                \begin{array}{c}
                                                                                  u \\
                                                                                  3 \\
                                                                                \end{array}
                                                                              \Big)-\Big(
                                                                                      \begin{array}{c}
                                                                                        t-u \\
                                                                                        3 \\
                                                                                      \end{array}
                                                                                    \Big)-\Big(
                                                                                            \begin{array}{c}
                                                                                              n-t \\
                                                                                              3 \\
                                                                                            \end{array}
                                                                                          \Big),$$
\end{multline*}
which is enough.
\end{proof}
\begin{Example}\label{315}
{\em Let $I=(x_1,x_2,x_3,x_4)\cap(x_3,x_4,x_5,x_6)\cap(x_5,x_6,x_1,x_2)$ then $u=2$, $t=4$ and by the above lemma we have
$\sdepth(I)\leq 3$.
}
\end{Example}
\begin{Proposition}\label{12}
Let $I=P_1\cap P_2 \cap P_3$, where ${P_1}=(x_1,\dots,x_t)$, ${P_2}=(x_{s+1},\dots,x_{r})$ and  ${P_3}=(x_{q+1},\dots,x_n,x_1,\dots,x_u)$ with $0<u\leq s\leq t\leq q\leq r\leq n$. Let $d=s-u+q-t+n-r$, $n-d\geq 3$. Then $$\sdepth(I)\leq 2+d+\frac{\Big(
                                                                  \begin{array}{c}
                                                                    n-d \\
                                                                    3 \\
                                                                  \end{array}
                                                                 \Big)-\Big(
                                                                                \begin{array}{c}
                                                                                  u \\
                                                                                  3 \\
                                                                                \end{array}
                                                                              \Big)-\Big(
                                                                                      \begin{array}{c}
                                                                                        t-s \\
                                                                                        3 \\
                                                                                      \end{array}
                                                                                    \Big)-\Big(
                                                                                            \begin{array}{c}
                                                                                              r-q \\
                                                                                              3 \\
                                                                                            \end{array}
                                                                                          \Big)}{\Big(
                                                                          \begin{array}{c}
                                                                            n-d \\
                                                                            2 \\
                                                                          \end{array}
                                                                        \Big)-\Big(
                                                                                \begin{array}{c}
                                                                                  u \\
                                                                                  2 \\
                                                                                \end{array}
                                                                              \Big)-\Big(
                                                                                      \begin{array}{c}
                                                                                        t-s \\
                                                                                        2 \\
                                                                                      \end{array}
                                                                                    \Big)-\Big(
                                                                                            \begin{array}{c}
                                                                                              r-q \\
                                                                                              2 \\
                                                                                            \end{array}
                                                                                          \Big)
                                                                 }.
$$
\end{Proposition}
\begin{proof}
Applying Corollary \ref{11} by recurrence on $d$-indeterminates  $\{x_{u+1},\dots,x_s,x_{t+1},\\\dots,x_{q},x_{r+1},\dots,x_n\}$ we get $\sdepth_S(I)\leq \sdepth_{S'}(I')+d$, where $$I'=(x_1,\dots,x_u,x_{s+1},\dots,x_t)\cap (x_{s+1},\dots,x_t,x_{q+1},\dots,x_r)\cap (x_{q+1},\dots,x_r,x_1,\dots,x_u)$$ and $S'=K[x_1,\dots,x_u,x_{s+1},\dots,x_t,x_{q+1},\dots,x_r]$. Now it is enough to apply Lemma \ref{8} to $I'$.

\end{proof}
Next we consider Case(2):
\begin{Theorem}\label{2}
Let $I=P_1\cap P_2\cap P_3$ where $P_1,P_2$ and $P_3$ are prime monomial ideals of $S$ such that $G({P_i})\cap G({P_j})=\emptyset$ for all $i\neq j $, $P_1+P_2+P_3=\mathfrak{m}$ and $\height(P_i)=d_i$. Then
\begin{multline*}
$$\sdepth(I)\leq 3+\frac{1}{d_1d_2d_3}
\Big[ \Big(
        \begin{array}{c}
          n \\
          4 \\
        \end{array}
      \Big)
-\sum\limits_{i=1}^3\Big(
        \begin{array}{c}
          d_i \\
          4 \\
        \end{array}
      \Big)
        - \sum\limits_{i=1}^3\Big(
                     \begin{array}{c}
                       d_i \\
                       3 \\
                     \end{array}
                   \Big)
(n-d_i)
 -\sum\limits_{i<j}\Big(
                \begin{array}{c}
                  d_i \\
                  2 \\
                \end{array}
              \Big)\Big(
                     \begin{array}{c}
                       d_j \\
                       2 \\
                     \end{array}
                   \Big)
\Big].
$$
\end{multline*}
\end{Theorem}
\begin{proof}
Note that $I$ is a squarefree monomial ideal generated in monomials of degree 3. Let $k:=\sdepth(I)$. The poset $P_{I}$ has a partition $\mathcal{P}$ : $P_{I}= \bigcup_{i=1}^{s}$ $[C_i,D_i]$, satisfying $\sdepth(\mathcal{D}(\mathcal{P}))=k$ where $\mathcal{D(P)}$ is the Stanley decomposition of $I$ with respect to the partition $\mathcal{P}$. For each interval $[C_i,D_i]$ in $\mathcal{P}$ with $|C_i|=3$ we have $|D_i|\geq k$ and note that there are $d_1d_2d_3$ such intervals. There are $|D_i|-|C_i|$ subsets of cardinality $4$ in this interval. Since these intervals are disjoint, counting the number of subsets of cardinality 4 we have
$$(d_1d_2d_3)(k-3)\leq \Big[ \Big(
        \begin{array}{c}
          n \\
          4 \\
        \end{array}
      \Big)
-\sum\limits_{i=1}^3\Big(
        \begin{array}{c}
          d_i \\
          4 \\
        \end{array}
      \Big)
        - \sum\limits_{i=1}^3\Big(
                     \begin{array}{c}
                       d_i \\
                       3 \\
                     \end{array}
                   \Big)
(n-d_i)
 -\sum\limits_{i<j}\Big(
                \begin{array}{c}
                  d_i \\
                  2 \\
                \end{array}
              \Big)\Big(
                     \begin{array}{c}
                       d_j \\
                       2 \\
                     \end{array}
                   \Big)
\Big], $$
which is enough.
\end{proof}
\begin{Example}\label{310}
{\em Let $I=(x_1^{a_1},\dots,x_5^{a_5})\cap (x_6^{a_6},\dots,x_{10}^{a_{10}})\cap (x_{11}^{a_{11}},\dots,x_{15}^{a_{15}})$ for some positive integers $a_i$. Then by Corollary \ref{1a} we have $\sdepth(I)\geq \lceil\frac{5}{2}\rceil+\lceil\frac{5}{2}\rceil+\lceil\frac{5}{2}\rceil=9$. Now by applying \cite[Corollary 2.2]{MI} and using Theorem \ref{2} we have $\sdepth(I)\leq 9$ which means that $\sdepth(I)=9$.
}
\end{Example}

Now we consider Case(3):
\begin{Proposition}\label{5}
Let ${P_1}=(x_1,\dots,x_t)$, ${P_2}=(x_{s+1},\dots,x_r)$, ${P_3}=(x_{q+1},\dots,x_n)$ with $s\leq t$, $q\leq r$ and $I=P_1\cap P_2\cap P_3$. Let
$$d:=\min \{\frac{2n+t-r-s+2}{2},\frac{n+r+s-q+2}{2}, \ n-\lfloor\frac{t}{2}\rfloor, \ n-\lfloor\frac{r-s}{2}\rfloor, \  n-\lfloor\frac{n-q}{2}\rfloor\}.$$ Then
$\sdepth(I)\leq d, \ if \ t\geq q$ and
$$\sdepth(I)\leq \min\{d, \ \frac{n+q-t+2}{2}\}, \ \ \ \ \ \ \ \ \ \ \ \ \ if \ \ \ t<q.$$
\end{Proposition}
\begin{proof} Let $v=x_n$ we see that $v\notin {I}$ and we have $I':={I}:v=(x_1,\dots,x_t)\cap (x_{s+1},\dots,x_r)$. Then by Proposition \cite[Proposition 1.3]{D2} we have $\sdepth_S({I})\leq \sdepth_S(I')$. We have $\sdepth_{S}({I'})=\sdepth(I'\cap K[x_1,\dots,x_r])+n-r$ by \cite[Lemma 3.6]{HVZ}. But by \cite[Proposition 2.13]{MI} and Theorem \ref{f} we have $$\sdepth(I'\cap K[x_1,\dots,x_r])\leq \min \{\frac{r+t-s+2}{2},r-\lfloor\frac{t}{2}\rfloor,r-\lfloor\frac{r-s}{2}\rfloor\}.$$ It follows   $$\sdepth_S(I)\leq\sdepth_S({I'})\leq \min \{\frac{2n+t-r-s+2}{2}, \ n-\lfloor\frac{t}{2}\rfloor, \ n-\lfloor\frac{r-s}{2}\rfloor\}.$$ Similarly if we take $v=x_1$ then we have $$\sdepth_S({I})\leq \min \{\frac{n+r+s-q+2}{2}, \ n-\lfloor\frac{r-s}{2}\rfloor, \ n-\lfloor\frac{n-q}{2}\rfloor\}.$$ and so $\sdepth_S(I)\leq d$.
If $t<q$ then take $v=x_{t+1}$. We have $I'':={I}:v=(x_1,\dots,x_t)\cap (x_{q+1},\dots,x_{n})$ and by \cite[Theorem 2.8]{MI} it follows $$\sdepth(I''\cap K[x_1,\dots,x_t,x_{q+1},\dots,x_n])\leq \frac{n-(q-t)+2}{2}$$ By \cite[Lemma 3.6]{HVZ} we get $$\sdepth_S({I})\leq \frac{n-(q-t)+2}{2}+q-t=\frac{n+q-t+2}{2}.$$
\end{proof}
\begin{Example}\label{317}
{\em Let ${I}=(x_1,\dots x_4)\cap (x_5,\dots,x_8)\cap (x_7,\dots,x_{10})$, that is $n=10$, $t=s=4$, $r=8$, $q=6$. Then $d=\min\{7,9,8,8,8\}=7$ and $\sdepth(I)\leq 7$ by the above proposition.
}
\end{Example}
One can extend Theorem \ref{2} for an intersection of arbitrary number of prime ideals. In this case the expression for the Stanley depth is more complicated. When there are four such prime ideals in the intersection we give an upper bound which is reasonable in some cases.\\
\indent Let $I=P_1\cap P_2\cap P_3\cap P_4$, where $P_1,P_2,P_3,P_4$ are prime ideals such that $G({P_i})\cap G({P_j})=\emptyset$ for all $i\neq j$ and $\height(P_i)=d_i$ where $d_1+d_2+d_3+d_4=n$. After renumbering the variables we assume that $d_1\geq d_2\geq d_3\geq d_4=d$ and ${P_i}=(x_{r_{i-1}+1},\dots,x_{r_{i}})$ where $r_i=d_1+\dots+d_i$.
\begin{Proposition}\label{3}
\begin{multline*}
$$\sdepth(I)\leq3+d+\\\frac{1}{d_{1}d_{2}d_{3}}
\Big[ \Big(
        \begin{array}{c}
          n-d \\
          4 \\
        \end{array}
      \Big)
-\sum\limits_{i=1}^3\Big(
        \begin{array}{c}
          d_{i} \\
          4 \\
        \end{array}
      \Big)
        - \sum\limits_{i=1}^3\Big(
                     \begin{array}{c}
                       d_{i} \\
                       3 \\
                     \end{array}
                   \Big)
(n-d_{i})
 -\sum\limits_{i<j}\Big(
                \begin{array}{c}
                  d_{i} \\
                  2 \\
                \end{array}
              \Big)\Big(
                     \begin{array}{c}
                       d_{j} \\
                       2 \\
                     \end{array}
                   \Big)
\Big].$$
\end{multline*}
\end{Proposition}
\begin{proof}
Define a map $\psi:S\longrightarrow S'=K[x_1,\dots,x_{n-1}]$ by $\psi(x_i)=x_i$ for $i\leq n-1$ and $\psi(x_n)=1$. Then we have $I'=\psi(I)={P_1}\cap {P_2}\cap {P_3}$ and by \cite[Lemma 2.2]{C1} it follows $\sdepth_S(I)\leq \sdepth_{S'}(I')+1$. Using Theorem \ref{2} we get
\begin{multline*}
$$\sdepth(I'\cap K[x_1,\dots,x_{n-d}])\leq \\3+\frac{1}{d_{1}d_{2}d_{3}}
\Big[ \Big(
        \begin{array}{c}
          n-d \\
          4 \\
        \end{array}
      \Big)
-\sum\limits_{i=1}^3\Big(
        \begin{array}{c}
          d_{i} \\
          4 \\
        \end{array}
      \Big)
        - \sum\limits_{i=1}^3\Big(
                     \begin{array}{c}
                       d_{i} \\
                       3 \\
                     \end{array}
                   \Big)
(n-d_{i})
 -\sum\limits_{i<j}\Big(
                \begin{array}{c}
                  d_{i} \\
                  2 \\
                \end{array}
              \Big)\Big(
                     \begin{array}{c}
                       d_{j} \\
                       2 \\
                     \end{array}
                   \Big)
\Big].$$
\end{multline*}
\\ Now by \cite[Lemma 3.6]{HVZ} $$\sdepth_S(I)\leq \sdepth(I'\cap K[x_1,\dots,x_{n-d}])+(d-1)+1,$$ which is enough.
\end{proof}
\begin{Example}\label{312}
{\em Let $I=(x_1^{a_1},\dots,x_5^{a_5})\cap (x_6^{a_6},\dots,x_{10}^{a_{10}})\cap (x_{11}^{a_{11}},\dots,x_{15}^{a_{15}})\cap (x_{16}^{a_{16}},x_{17}^{a_{17}})$ for some positive integers $a_i$. Applying \cite[Corollary 2.2]{MI} and using Proposition \ref{3} we have $\sdepth(I)\leq 11$. By Corollary \ref{1a} $\sdepth(I)\geq 10.$
}
\end{Example}

\end{document}